\documentclass[12pt,reqno]{amsart}
\usepackage{template}  

\usetkzobj{all}

\definecolor{darkblue}{HTML}{004B83}
\definecolor{darkgreen}{HTML}{0C8900}

\begin{document}
  \title{Uniqueness of Optimal Point Sets Determining Two Distinct Triangles}
  \author[H. N. Brenner] {Hazel N. Brenner}
  \email{\textcolor{blue}{\href{mailto:hazelbrenner@vt.edu}{hazelbrenner@vt.edu}}}
  \address{Department of Mathematics, Virginia Tech, Blacksburg, VA 24061}
  
	\author[J. S. Depret-Guillaume]{James S. Depret-Guillaume}
	\email{\textcolor{blue}{\href{mailto:jdg@vt.edu}{jdg@vt.edu}}}
	\address{Department of Mathematics, Virginia Tech, Blacksburg, VA 24061}  
  
  \author[E. A. Palsson]{Eyvindur A. Palsson}
	\email{\textcolor{blue}{\href{mailto:palsson@vt.edu}{palsson@vt.edu}}}
  \address{Department of Mathematics, Virginia Tech, Blacksburg, VA 24061}
  
  \author[S. Senger]{Steven Senger}
  \email{\textcolor{blue}{\href{mailto:stevensenger@missouristate.edu}{stevensenger@missouristate.edu}}}
  \address{Department of Mathematics, Missouri State University, Springfield, MO, USA.}

  \subjclass[2010]{52C10 (primary), 52C35 (secondary)}
  \keywords{Distinct triangles, Erd\H{o}s problem, Optimal configurations, Finite point configurations}
  \date{\today}
  \thanks{The work of the third listed author was supported in part by Simons Foundation Grant \#360560.}
  
  \begin{abstract}
			In this paper, we show that the maximum number of points in $d\geq3$ dimensions determining exactly 2 distinct triangles is $2d$. We further show that this maximum is uniquely achieved by the vertices of the $d$-orthoplex. We build upon the work of Hirasaka and  Shinohara who determined in \cite{hirasaka_shinohara} that the $d$-orthoplex is such an optimal configuration, but did not prove its uniqueness. Further, we present a more elementary argument for its optimality.
  \end{abstract}


  \maketitle
  
  \tableofcontents
    
  \section{Introduction}
		Paul Erd\H{o}s posed his distinct distance conjecture in 1946, spawning a wide array of similar problems in discrete geometry. He originally conjectured that a set of $n$ points in general position in the plane must determine $\Omega(n/\sqrt{\text{log}(n)})$ distinct distances \cite{erdos}. This conjecture was ultimately proved by Guth and Katz in 2015 \cite{guth_katz}. Instead of fixing the number of points, consider fixing the number of distinct distances $k$. A question that is closely associated to the original Erd\H{o}s distinct distance problem asks: what is the maximal number of points determining $k$ distinct distances? In 1996, Erd\H{o}s and Fishburn found all the optimal point configurations in the plane determining $k\leq 4$ distinct distances and found an optimal point configuration determining $k=5$ distinct distances (see Figure \ref{fig:distances}). 
	\begin{figure}[!h]
		\includegraphics[scale=0.8]{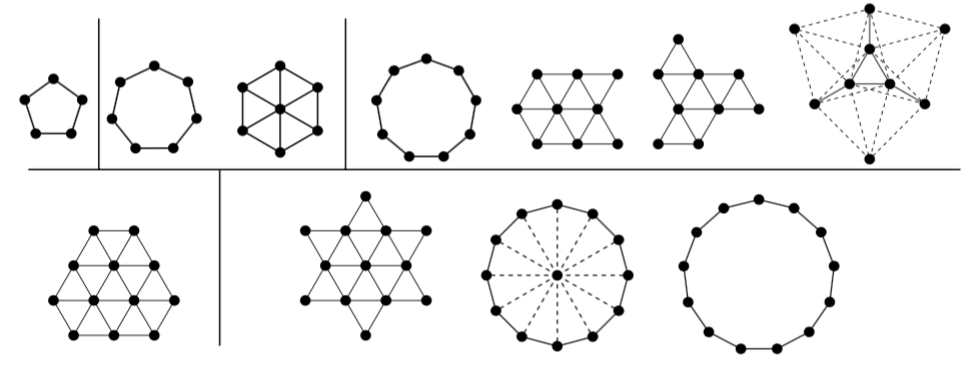}
		\caption{Maximal configurations determining exactly $k$ distances, for $2 \leq k \leq 6$ \cite{brass_et_al}.  Note that for each $k>2$, there is an example from the triangular lattice; it is conjectured that this is always the case for $k$ large enough.}
		\label{fig:distances}
	\end{figure}		
		They also conjectured that for each $k\geq 3$, at least one solution lies on the triangular lattice, and for $k\geq 7$, all solutions lie on the triangular lattice. This conjecture remains open \cite{erdos_fishburn}. 
 		
     A distance can be interpreted as a 1-simplex, so a natural generalization of this problem is to find the greatest number of points determining $k$ higher simplices. In our case, we consider 2-simplices, i.e. triangles. Some work has already been done in this direction. In \cite{epstein_et_al}, Epstein et al. showed that the optimal configuration determining one distinct triangle is the vertices of the square, and the optimal configurations determining two distinct triangles are the square with its center and the vertices of the regular pentagon. In \cite{brenner_et_al}, Brenner et al. determine that the unique optimal configuration determining one distinct triangle is the $d$-simplex in $\R^d$ for $d\geq 3$. Finally, in \cite{hirasaka_shinohara}, Hirasaka and Shinohara determined that an optimal configuration determining two triangles is the $d$-orthoplex in $\R^d$ for $d\geq 3$. As our main result, we provide a more elementary argument that the $d$-orthoplex is an optimal configuration determining two distinct triangles and offer a proof that it is in fact the unique optimal configuration determining two distinct triangles.
  \begin{thm}\label{thm:main_result}
  	The vertices of the $d$-orthoplex are the unique optimal configuration determining two distinct triangles in $\R^d$ for $d\geq 3$.    
  \end{thm}
  
  \section{Definitions, Lemmas and Setup}
  First, let us formalize our notion of distinct triangles.
  
  \begin{defi}
	  Given a finite point set $P \subset \R^d$, we say two triples $(a,b,c), (a',b',c') \in P^3$ are equivalent if there is an isometry mapping one to the other, and we denote this as $(a,b,c) \sim (a',b',c')$.
	\end{defi}
	
	\begin{defi}
	  Given a finite point set $P \subset \R^d$, we denote by $P_{nc}^3$ the set of noncollinear triples $(a,b,c) \in P^3$.
	\end{defi}
	
	\begin{defi}
	  Given a finite point set $P \subset \R^d$, we define the set of distinct triangles determined by $P$ as
	  \begin{equation}
	    T(P) := P_{nc}^3 / \sim.
	  \end{equation}
	  \label{def:tris}
	\end{defi}
	
	We will need the following lemmas to proceed with our argument for our main result. Note that given a finite point set $S$ to say that a point $P$ in $S$ determines a distance $d$ means that there is another point $Q$ in $S$ such that $PQ = d$.
	
	\begin{restatable}{lem}{lemmaxdds}
		A finite point set determining $t$ distinct triangles determines at most $2t + 1$ distinct distances.	
		\label{lem:max_dds}
	\end{restatable}
	
	\begin{restatable}{lem}{lemdistinctbound}
		A finite point set containing a point which determines $n$ distinct distances must determine at least $$\ceil*{\frac{n}{2}\cdot\floor*{\frac{n}{3}}}$$
		distinct triangles.
		\label{lem:distinct_bound}
	\end{restatable}
	
	\begin{restatable}{lem}{lemrepeatbound}
		Given a finite point set $S$ and a point $P$ in $S$. Let $n$ be the number of distinct distances determined by $P$. Let $m$ denote the number of these distinct distances which are determined by $P$ and two or more distinct points. Then, $S$ determines at least $$\ceil*{\frac{n(n-1) + 4m}{6}}$$ distinct triangles.
    \label{lem:repeat_bound}
	\end{restatable}
	
	\noindent For an illustration of Lemma \ref{lem:repeat_bound}, see Figure \ref{fig:repeat_bound}.
	\begin{figure}[!h]
		\centering
		\begin{tikzpicture}[scale=.02]
	\tkzDefPoint(0,0){O}
	\tkzDefPoint(-90, 40){A}
	\tkzDefPoint(-55, 80){B} %
	\tkzDefPoint(0,100){C}
	\tkzDefPoint(55, 80){D}
	\tkzDefPoint(90, 40){E}
	
	\tkzDrawPoints(O, A, B, C, D, E)
	\tkzDrawSegment(O, A)	
	\tkzDrawSegment[color=darkblue](O, B)
	\tkzDrawSegment[color=darkgreen](O, C)
	\tkzDrawSegment[color=darkblue](O, D)
	\tkzDrawSegment[color=darkgreen](O, E)
	
	\tkzMarkSegment[size=2pt, mark=|](O, A)
	\tkzMarkSegment[size=2pt, mark=||](O, B)
	\tkzMarkSegment[size=2pt, mark=|||](O, C)
	\tkzMarkSegment[size=2pt, mark=||](O, D)
	\tkzMarkSegment[size=2pt, 	mark=|||](O, E)
	 
 	\tkzLabelPoint(O){$\Oc$}
 	\tkzLabelPoints[left](A, B, C, D)
 	\tkzLabelPoints[above left](E)
	
\end{tikzpicture}	
		\caption{The distinct distances present are $d_1, d_2$ and $d_3$, so $n=3$. Of these, $d_2$ and $d_3$ are repeated, so $m=2$. Thus, this configuration determines at least 3 distinct triangles.}
		\label{fig:repeat_bound}
	\end{figure}
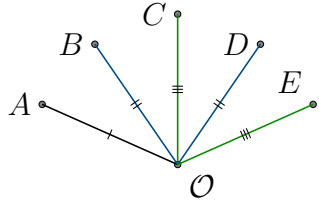	
	 
	 The structure of our main proof will be based primarily on Lemma \ref{lem:max_dds}. By this lemma, configurations determining only two distinct triangles (our primary focus in this paper) can determine at most five distinct distances. In Section \ref{sec:elim_dds} we split into cases and show that any configuration of at least $2d$ points determining five, four or three distinct distances must determine more than two distinct triangles in $d$ dimensions. Trivially, any finite point set determining only one distinct distance can determine at most one distinct triangle. So we can conclude that any configuration of at least $2d$ points determining two distinct triangles must determine exactly two distinct distances. Then, in Section \ref{sec:geom}, we prove that any configuration of at least $2d$ points determining only two distinct triangles must, in fact, determine triangles of specific geometry. That is, one of the triangles must be equilateral and the other must be isosceles with its repeated edge length being that of the equilateral triangle. Then, in Section \ref{sec:main_result}, we show using these restrictions that any set of $2d$ points in $\R^d$ determining two distinct triangles must be the vertices of a $d$-orthoplex. As a corollary, we then show that a configuration of more points necessarily determines more distinct triangles. This proves Theorem \ref{thm:main_result}.
 
 	\section{Eliminating Higher Numbers of Distinct Distances}\label{sec:elim_dds}
 	All of the arguments below will proceed with the following setup. Given a set $S$ of $2d$ points $\R^d$, label one of the points of $S$ as $\Oc$ and the remaining points $P_1$, $P_2$, $\dots$, $P_{2d - 1}$. We will often choose $\Oc$ to be some convenient point in a manner that clearly does not lose generality by relabeling. 
 	\begin{prop}\label{prop:elim_dds}
 		Any configuration of at least $2d$ points in $d\geq 3$ dimensions determining five, four or three distinct distances determines more than two distinct triangles.
 	\end{prop}
 	\begin{proof}\
 		\begin{description}[align=left]
 			\item[5 Distinct Distances] Given that the two triangles must share an edge, two distinct triangles determining five distinct distances must both be scalene. This means that none of the distances from $\Oc$ to the remaining points can be repeated, otherwise the triangle formed by $\Oc$ and those two points would be isosceles (or equilateral). For $d>3$, $\Oc$ determines more than five distances, so by the pigeonhole principle, at least one of the distances must be repeated, which is a contradiction. In $d=3$ we clearly may assume all of the distances from $\Oc$ are distinct. By Lemma \ref{lem:distinct_bound}, these distances then determine at least $\ceil*{\frac{5}{3}\cdot\floor*{\frac{5}{2}}} = 4$ distinct triangles.
 			\item[4 Distinct Distances] By similar reasoning, a configuration determining an equilateral triangle or two distinct isosceles triangles can determine at most three distinct distances. So, a configuration determining two distinct triangles and four distinct distances may determine no equilateral triangles and at most one distinct isosceles triangle. Note that when a given distance is determined by a single fixed point $P$ and at least two other distinct points $A$, $B$, etc., $\triangle PAB$ must be either equilateral or isosceles. Since it is impossible for one of the two distinct triangles to be equilateral in this case, we can assume that any repeated distances determine isosceles triangles. Since two distinct repeated distances necessarily form two non-congruent isosceles triangles and only one of the distinct triangles may be isosceles, only one distinct distance may be repeated. Finally, consider a configuration with at least two repetitions of a single distinct distance, e.g., $\Oc P_1 =\Oc P_2 = \Oc P_3= d_1$.
 			  \begin{figure}[h!]
 				  \centering
 				  \begin{tikzpicture}[scale=.02]
	\tkzDefPoint(0,0){O}
	\tkzDefPoint(-55, 65){A}
	\tkzDefPoint(0,95){B}
	\tkzDefPoint(55, 65){C}
	
	\tkzDrawPoints(O, A, B, C)	
	\tkzDrawSegment(O, A)	
	\tkzDrawSegment(O, B)
	\tkzDrawSegment(O, C)
	
	\tkzDrawSegment[color=darkblue](A, B)
	\tkzDrawSegment[color=darkblue](B, C)
	\tkzDrawSegment[color=darkblue](A, C)
	
	\tkzMarkSegment[size=2pt, mark=|](O, A)
	\tkzMarkSegment[size=2pt, mark=|](O, B)
	\tkzMarkSegment[size=2pt, mark=|](O, C)
	\tkzMarkSegment[size=2pt, mark=||](A, B)
	\tkzMarkSegment[size=2pt, mark=||](B, C)
	\tkzMarkSegment[size=2pt, mark=||, pos=0.6](A, C)
	 
	\tkzLabelPoint(O){$\Oc$}
	\tkzLabelPoint[left](A){$P_1$}
	\tkzLabelPoint[above](B){$P_2$}
	\tkzLabelPoint[right](C){$P_3$}

\end{tikzpicture}
 				  \caption{}
 				  \label{fig:icecream_cone}
        \end{figure}
        Since this configuration cannot determine an equilateral triangle and determines at most one distinct isosceles triangle, assume that $\triangle \Oc P_1 P_2$, $\triangle \Oc P_2 P_3$ and $\triangle \Oc P_1 P_3$ are all congruent to the same isosceles triangle. Then $\triangle P_1 P_2 P_3$ is equilateral (see Figure \ref{fig:icecream_cone}). So, by contradiction, the repeated distinct distance may be repeated only once. Since only one of the distances from $\Oc$ to the other $2d -1$ points may be repeated at most once, all four distinct distances must appear. So by Lemma \ref{lem:distinct_bound}, these distances then determine at least $\ceil*{\frac{4}{3}\cdot\floor*{\frac{4}{2}}} = 3$ distinct triangles. Notice that this argument depends only on having six or more points, so this holds in all cases.  
 			\item[3 Distinct Distances] In a configuration of points determining three distinct distances and two distinct triangles, at most one of the triangles may be equilateral, in which case the other must be scalene.
 			  \begin{rek*}
 			  	In such a configuration of points, any distance repeating more than once from a single point must determine an equilateral triangle.
 			  \end{rek*}
 			  To see that this is true, consider a configuration in which $\Oc P_1 =\Oc P_2 = \Oc P_3= d_1$. Clearly, if only one or two of the three incomplete triangles is equilateral, we would have an equilateral and an isosceles triangle co-occurring. Otherwise, if none of the distances $P_1 P_2, P_2 P_3$ or $P_3 P_1$ are equal to $d_1$, and any of them differ, then $\triangle P_1 P_2 P_3$ must be an isosceles triangle not congruent to either of the two isosceles triangles containing $\Oc$. This is a contradiction as this configuration would then determine three distinct triangles. So, they must all be either $d_2$ or $d_3$. Then $\triangle P_1 P_2 P_3$ is equilateral, which is impossible as we already have an isosceles triangle. So, we must have $P_1 P_2 = P_2 P_3 = P_3 P_1 = d_1$.
 			  
 			  From this, it is clearly impossible for all of the distances from $\Oc$ to be the same, otherwise every triangle determined by this configuration would be equilateral, so specifically, the configuration would only determine one distinct triangle. So, there must be at least two distinct distances determined from $\Oc$. Consider one such configuration in which $\Oc P_1 =\Oc P_2 = \Oc P_3= d_1$ while $\Oc P_4 = d_2$. By the remark, $\triangle \Oc P_2 P_3$ must be equilateral, so specifically $P_2 P_3=d_1$. Since the remaining distinct triangle must be scalene, $\Oc P_3 P_4$ must be scalene, so specifically $P_3 P_4 = d_3$. By the same reasoning, $\triangle P_2 P_3 P_4$ must also be a congruent scalene triangle, so $P_2 P_4 = d_2$. Then $\triangle \Oc P_2 P_4$ is an isosceles triangle, which is a contradiction since $\triangle \Oc P_2 P_3$ is equilateral
 			  \begin{figure}[h!]
 				  \centering
 				  \begin{tikzpicture}[scale=.02]
	\tkzDefPoint(0,0){O}
	\tkzDefPoint(-90, 40){A}
	\tkzDefPoint(-55, 80){B} %
	\tkzDefPoint(0,100){C}
	\tkzDefPoint(55, 80){D}
	\tkzDefPoint(90, 40){E}
	
	\tkzDefMidPoint(O, E) \tkzGetPoint{M}

	\tkzDrawPoints(O, A, B, C, D, E)
	\tkzDrawSegment(O, A)	
	\tkzDrawSegment(O, B)
	\tkzDrawSegment(O, C)
	\tkzDrawSegment[color=darkgreen](O, D)
	\tkzDrawSegment[dashed](O, E)
	\tkzDrawSegment(B, C)
	\tkzDrawSegment[color=darkgreen](B, D)
	\tkzDrawSegment[color=darkblue](C, D)	
	
	\tkzMarkSegment[size=2pt, mark=|](O, A)
	\tkzMarkSegment[size=2pt, mark=|](O, B)
	\tkzMarkSegment[size=2pt, mark=|](O, C)
	\tkzMarkSegment[size=2pt, mark=||](O, D)
	\tkzMarkSegment[size=2pt, mark=|](B, C)
	\tkzMarkSegment[pos=0.6, size=2pt, mark=||](B, D)
	\tkzMarkSegment[size=2pt, mark=|||](C, D)
	 
 	\tkzLabelPoint(O){$\Oc$}
 	\tkzLabelPoint[left](A){$P_1$}
 	\tkzLabelPoint[left](B){$P_2$}
 	\tkzLabelPoint[above](C){$P_3$}
 	\tkzLabelPoint[right](D){$P_4$}
 	\tkzLabelPoint[right](E){$P_5$}
	
	\tkzLabelPoint[anchor=center](M){?}
\end{tikzpicture}
 				  \caption{}
 				  \label{fig:3distcontra}
        \end{figure} 
        
       Thus, it is impossible for any distance to repeat more than once without determining more than two distinct triangles. For $d>3$, by a simple application of the pigeonhole principle it is clear that one of the distances must be repeated more than once, so by the above, configurations of more than six points determining three distinct distances necessarily determine at least three distinct triangles independent of dimension. In the case where $d=3$, the only possible configuration contains one repeat of two of the distinct distances (for example $d_1$ and $d_2$) and one occurrence of the third. This configuration determines $n=3$ distinct distances and two of them are repeated (i.e., $m=2$). So, by \ref{lem:repeat_bound}, this configuration must determine at least $\ceil*{\frac{3(3-1) + 4(2)}{6}} = 3$ distinct triangles.
 		\end{description}
 	\end{proof}
 	
 	\section{Geometry of Two Distinct Triangles Determined by \texorpdfstring{$2d$}{2d} Points}\label{sec:geom}
 	By Proposition \ref{prop:elim_dds}, any configuration of at least $2d$ points in $d\geq 3$ dimensions determining two distinct triangles must also determine only two distinct distances. We can then prove the following proposition about the geometry of two distinct triangles.
 	
 	\begin{prop}\label{prop:tri_geom}
 		Given a configuration of at least $2d$ points in $d\geq 3$ dimensions determining two distinct triangles and two distinct distances, one of the triangles must be equilateral, and the other triangle must be isosceles with its repeated edge length being the same as the edge length of the equilateral triangle.
 	\end{prop}
 	\begin{proof}
 		We will begin by proving that one of the triangles must be equilateral. Assume to the contrary that neither of the triangles is equilateral. In this case, both of the triangles are isosceles with side lengths $(d_1, d_1, d_2)$ and $(d_2, d_2, d_1)$, which we will reference as $T_1$ and $T_2$ respectively. Clearly, since this configuration determines two distinct distances, there must be at least one point which determines both distances. Additionally, by the pigeonhole principle, one of the distances must occur at least $d$ out of $2d-1$ times from a given point. So, without loss of generality, assume $d_1$ is determined by $\Oc$ and at least three distinct points $P_1$, $P_2$, $P_3$, etc. All of $\triangle \Oc P_1 P_2$, $\triangle \Oc P_1 P_3$ and $\triangle \Oc P_2 P_3$ must be congruent to $T_1$. So, $P_1 P_2 = P_1 P_3 = P_2 P_3 = d_2$. Then $\triangle P_1 P_2 P_3$ is equilateral, which is a contradiction. 
 		
 		Since one of the triangles must be equilateral (assume without loss of generality that it has side length $d_1$), it remains to eliminate the case where repeated distance of the second triangle is $d_2$. We will denote the equilateral triangle as $T_{equ}$ and the isosceles triangle as $T_{iso}$. For the following, we will split into two cases. For the first case we will consider $d_1$ as the distance that occurs $d$ times. Without loss of generality, assume that $\Oc P_1$,$\dots$ $\Oc P_d$ are $d_1$ and assume $\Oc P_{d+1} = d_2$.
 		\begin{figure}[h!]
 		  \centering
 		  \begin{tikzpicture}[scale=.02]
	\tkzDefPoint(0,0){O}
	\tkzDefPoint(-90, 40){A}
	\tkzDefPoint(-55, 80){B}
	\tkzDefPoint(0,100){C}
	\tkzDefPoint(55, 80){D}
	\tkzDefPoint(90, 40){E}
	
	\tkzDefMidPoint(O, E) \tkzGetPoint{M}

	\tkzDrawPoints(O, A, B, C, D, E)
	\tkzDrawSegment(O, A)	
	\tkzDrawSegment(O, B)
	\tkzDrawSegment(O, C)
	\tkzDrawSegment[color=darkblue](O, D)
	\tkzDrawSegment[dashed](O, E)
	
	\tkzMarkSegment[size=2pt, mark=|](O, A)
	\tkzMarkSegment[size=2pt, mark=|](O, B)
	\tkzMarkSegment[size=2pt, mark=|](O, C)
	\tkzMarkSegment[size=2pt, mark=||](O, D)

 	\tkzLabelPoint(O){$\Oc$}
 	\tkzLabelPoints[left](A, B, C, D)
 	\tkzLabelPoints[above](E)
 	
 	\tkzLabelPoint[anchor=center](M){?}

\end{tikzpicture}
 			\caption{The setup for the 3-dimensional version of this argument}
    \end{figure}
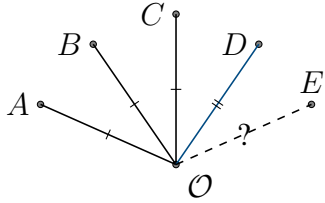
 		Notice that a single edge length of $d_2$ is sufficient to show that a given triangle is congruent to $T_{iso}$, and similarly two edge lengths of $d_1$ are sufficient to show congruence to $T_{equ}$. Then, since each triangle in $\{\triangle \Oc P_i P_j\}_{1 \leq i, j\leq d}$ for $i\neq j$ contains two edges of length $d_1$ (namely $\Oc P_i$ and $\Oc P_j$), each must be congruent to $T_{equ}$. And, specifically, we must have $P_i P_j = d_1$. This means that $\{\Oc, P_1,\dots, P_d\}$ is a set of $d+1$ mutually equidistant points in $d$ dimensions; namely, it determines a $d$-simplex. Notice now that none of the remaining points $P_{d+2},\dots,P_{2d - 1}$ may lie at distance $d_1$ from any of the points of the simplex, else, it must lie at distance $d_1$ from all of the points of the simplex. This is a contradiction as the maximum number of mutually equidistant points in $d$-dimensional Euclidean space is $d+1$, as is achieved by the regular $d$-simplex. 
 		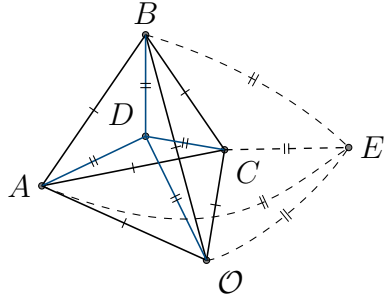
\begin{figure}[h!]
 		  \centering
 		  \begin{tikzpicture}[scale=.03]
	\tkzDefPoint(27,0){O}
	\tkzDefPoint(-46, 33){A}
	\tkzDefPoint(0, 100){B} 
	\tkzDefPoint(35,49){C}
	\tkzDefPoint(0, 55){D}
	\tkzDefPoint(90, 50){E}
	

	\tkzDrawPoints(O, A, B, C, D, E)
	\tkzDrawSegment(O, A)	
	\tkzDrawSegment(O, B)
	\tkzDrawSegment(O, C)
	\tkzDrawSegment(A, B)
	\tkzDrawSegment(B, C)
	\tkzDrawSegment(A, C)

	\tkzDrawSegment[color=darkblue](D, O)
	\tkzDrawSegment[color=darkblue](D, A)
	\tkzDrawSegment[color=darkblue](D, B)
	\tkzDrawSegment[color=darkblue](D, C)
	
 	\draw[dashed, decorate, decoration={markings, mark=at position 0.49 with {\arrow[line width=.08mm]{|}};}] (B) .. controls (35, 90) and (70, 70) .. (E);
 	\draw[dashed, decorate, decoration={markings, mark=at position 0.51 with {\arrow[line width=.08mm]{|}};}] (B) .. controls (35, 90) and (70, 70) .. (E);
 	\draw[dashed] (B) .. controls (35, 90) and (70, 70) .. (E);
 	
 	\tkzDrawSegment[dashed, line width=.5pt](C, E)
 	
 	\draw[dashed, decorate, decoration={markings, mark=at position 0.49 with {\arrow[line width=.08mm]{|}};}] (O) .. controls (35, 0) and (70, 20) .. (E);
 	\draw[dashed, decorate, decoration={markings, mark=at position 0.51 with {\arrow[line width=.08mm]{|}};}] (O) .. controls (35, 0) and (70, 20) .. (E);
 	\draw[dashed] (O) .. controls (35, 0) and (70, 20) .. (E);
 	
 	\draw[dashed, decorate, decoration={markings, mark=at position 0.7 with {\arrow[line width=.1mm]{|}};}] (A) .. controls (35, 0) and (60, 30) .. (E);
 	\draw[dashed, decorate, decoration={markings, mark=at position 0.69 with {\arrow[line width=.1mm]{|}};}] (A) .. controls (35, 0) and (60, 30) .. (E);
 	\draw[dashed] (A) .. controls (35, 0) and (60, 30) .. (E);

	\tkzMarkSegment[size=2pt, mark=|](O, A)
	\tkzMarkSegment[size=2pt, mark=|](O, B)
	\tkzMarkSegment[size=2pt, mark=|](O, C)
	\tkzMarkSegment[size=2pt, mark=|](A, B)
	\tkzMarkSegment[size=2pt, mark=|](B, C)
	\tkzMarkSegment[size=2pt, mark=|](A, C)
	
	\tkzMarkSegment[size=2pt, mark=||](D, O)
	\tkzMarkSegment[size=2pt, mark=||](D, A)
	\tkzMarkSegment[size=2pt, mark=||](D, B)
	\tkzMarkSegment[size=2pt, mark=||](D, C)
	
	\tkzMarkSegment[size=2pt, mark=||](C, E)
	

 	\tkzLabelPoint(O){$\Oc$}
 	\tkzLabelPoints[left](A)
 	\tkzLabelPoints[above](B)
 	\tkzLabelPoints[below right](C)
 	\tkzLabelPoints[above left](D)
 	\tkzLabelPoints[right](E)

\end{tikzpicture}
 			\caption{In the 3-dimensional case, our simplex is the tetrahedron whose vertices are $\Oc$, $A$, $B$ and $C$. The point $D$ is the center of the tetrahedron. It is intuitively clear from this image that there is no consistent position for the contradictory point $E$ other than $D$.}
    \end{figure}
    So, specifically, every distance from each point that is not a vertex of the simplex to a vertex of the simplex must be $d_2$. So each point $P_{d+1},\dots, P_{2d -1}$ is equidistant from all of the vertices of the simplex. In general, $d+1$ points in general linear position uniquely determine a $(d-1)$-sphere (see Appendix \ref{app:geom}). Notice then that since the vertices of a $d$-simplex lie in general position, they uniquely determine a $(d-1)$-sphere. But, since each of the points $P_{d+1},\dots, P_{2d -1}$ by definition lie at the center of a $(d-1)$-sphere containing all of the vertices of the simplex, and there is only one such sphere, all of these points must coincide, which is a contradiction. Note that in $d=3$, there are two such points, so this argument holds for all $d\geq 3$.
    
    Consider instead the case where $d_2$ is repeated $d$ times. For clarity, relabel $P_{d+1}$ as $Q$ and without loss of generality, let $\Oc Q = d_1$. By a similar argument to the $d_1$ case, each triangle in $\{\triangle \Oc P_i P_j\}_{1 \leq i, j\leq d}$ for $i\neq j$ is congruent to $T_{iso}$. Then, specifically, each distance $P_i P_j$ is $d_1$. Clearly then, $\{P_1, \dots, P_d\}$ is a set of $d$ mutually equidistant point. So, these points are the vertices of a $(d-1)$-simplex. Now, notice that since $\Oc Q = d_1$ and each distance from $\Oc$ to a vertex of the simplex (i.e. $P_i$ for $1\leq i\leq d$) must be $d_2$, $\triangle \Oc Q P_i$ is congruent to $T_{iso}$. Thus $Q P_i = d_2$ for all $1 \leq i \leq d$. The set of all points equidistant from two points in $d$ dimensions, is a $(d-1)$-hyperplane (see Appendix \ref{app:geom}) So specifically, the vertices of the simplex lie on a $(d-1)$-hyperplane orthogonal to the line through $\Oc$ and $Q$.
    
    Suppose $\Oc P_{d+2} = d_2$. Then $\triangle \Oc P_{d+2} P_i$ must be congruent to $T_{iso}$ for $1\leq i\leq d$, so specifically the distance from $P_{d+2}$ to every vertex of the simplex must be $d_1$. So, $\{P_1, \dots, P_d, P_{d+2}\}$ is a set of $d+1$ mutually equidistant points, i.e. the vertices of a $d$-simplex. Similarly, since $QP_1 = d_2$ and $P_1 P_{d+2} = d_1$, we must have $QP_{d+2} = d_2$. Then, $P_{d+2}$ is equidistant from $\Oc$ and $Q$, so it must lie in the same $(d-1)$-hyperplane as the vertices of the simplex. This is a contradiction as a $d$-simplex clearly cannot lie within a $(d-1)$-hyperplane.
    
    So, fixing $d<j\leq 2d-1$, each $\Oc P_j$ must be $d_1$. Then each triangle made up of either $\Oc$ or $Q$, one of the vertices of the simplex and one of the remaining points $P_j$, must be congruent $T_{iso}$. So specifically, the distance between any point of the simplex and each of the remaining points must be $d_2$. Additionally, since $Q P_1 = P_1 P_j = d_2$, we must have $Q P_j = d_1$. Notice then that all such points $P_j$ are equidistant from $\Oc$ and $Q$. Namely they lie in the same $(d-1)$-hyperplane as the vertices of the simplex. Notice that each of these points is also equidistant from all of the vertices of the simplex. But as seen earlier, there is a unique point equidistant from all of the vertices of a $d$-simplex in $d$-dimensional Euclidean space. So, all of these points must coincide, which is a contradiction. So a point configuration in $d$ dimensions can only determine two distinct triangles if one of them is equilateral and the other is an isosceles triangle whose repeated edge length is the same as the side length of the equilateral triangle.
 	\end{proof}
 	
 	\section{Proving the Main Result}\label{sec:main_result}
 	In this section, we seek to prove that the $d$-orthoplex is the unique configuration of $2d$ points determining two distinct triangles in $d\geq 3$ dimensions. Notice that all previous sections assumed only that we have at least $2d$ points in $d$ dimensions, so at the end of this section, as an easy corollary to the proof of Proposition \ref{prop:unique}, we will see that any configuration of strictly more than $2d$ points in $d\geq 3$ dimensions determines more than two distinct triangles. Taking Proposition \ref{prop:unique} together with Corollary \ref{cor:optimal} will yield our main result, Theorem \ref{thm:main_result}.
 	
 	Note that in the following argument, we need only consider the pair of triangles determined by Proposition \ref{prop:tri_geom}. For what follows, label the equilateral triangle $T_{equ}$ and the isosceles triangle $T_{iso}$. This will be enough to directly construct the vertices of the $d$-orthoplex in $d$ dimensions. The proof proceeds inductively, so we begin by proving the following proposition which will serve as our base case.
 	\begin{prop}
		The set of vertices of an octahedron is the unique configuration of six points determining two distinct triangles in three dimensions.
 	\end{prop}
 	\begin{proof}
 		We begin with a similar setup as before, noting that at least one of the six points must determine both distances so we may assume by relabeling that $\Oc$ is one of these points. For convenience, we will label the remaining points $A$, $B$, $C$, $D$ and $E$. Assume $\Oc E = d_2$. Then, notice that $\triangle \Oc AE$, $\triangle \Oc BE$, $\triangle \Oc CE$ and $\triangle \Oc DE$ must all be congruent to $T_{iso}$, meaning all distances from points other than $E$ to $\Oc$ and all distances from points other than $\Oc$ to $E$ must be $d_1$.
		\begin{figure}[h!]
 		  \centering
			\begin{tikzpicture}[scale=.02]
	\tkzDefPoint(0,0){O}
	\tkzDefPoint(-90, 40){A}
	\tkzDefPoint(-55, 80){B} %
	\tkzDefPoint(0,100){C}
	\tkzDefPoint(55, 80){D}
	\tkzDefPoint(90, 40){E}
	
	\tkzDrawPoints(O, A, B, C, D, E)
	\tkzDrawSegment(O, A)	
	\tkzDrawSegment(O, B)
	\tkzDrawSegment(O, C)
	\tkzDrawSegment(O, D)
	\tkzDrawSegment(E, A)	
	\tkzDrawSegment(E, B)
	\tkzDrawSegment(E, C)
	\tkzDrawSegment(E, D)
	\tkzDrawSegment[color=darkblue](O, E)
	
	\tkzMarkSegment[size=2pt, mark=|](O, A)
	\tkzMarkSegment[size=2pt, mark=|, pos=0.7](O, B)
	\tkzMarkSegment[size=2pt, mark=|](O, C)
	\tkzMarkSegment[size=2pt, mark=|, pos=0.6](O, D)
	\tkzMarkSegment[size=2pt, mark=||](O, E)
	
	\tkzMarkSegment[size=2pt, mark=|, pos=0.8](E, A)
	\tkzMarkSegment[size=2pt, mark=|](E, B)
	\tkzMarkSegment[size=2pt, mark=|, pos=0.55](E, C)
	\tkzMarkSegment[size=2pt, mark=|](E, D)
	 
 	\tkzLabelPoint(O){$\Oc$}
 	\tkzLabelPoints[left](A, B)
 	\tkzLabelPoints[above](C)
 	\tkzLabelPoints[right](D, E)
	
\end{tikzpicture}
			\caption{}
    \end{figure} 		
 		  Notice that $A$, $B$, $C$ and $D$ are then all equidistant from $\Oc$ and $E$, so specifically they must be coplanar and lie on a circle centered on the midpoint of $\Oc$ and $E$ (see Appendix \ref{app:geom}). So, to determine the remaining distances, it remains only to determine the cyclic quadrilateral $\square ABCD$. Notice that any two adjacent edges equal to $d_2$ would necessarily determine an isosceles triangle with $d_2$ as the repeated edge, which is not one of our two distinct triangles. So, the only two cases we must eliminate are an isosceles trapezoid with three edges equal to $d_1$ and one edge equal to $d_2$, and a rectangle with opposite pairs of edges equal to $d_1$ and $d_2$. The diagonal of a rectangle is always longer than either of the two sides, so $AC$ and $BD$ could not be one of the two distances, which is a contradiction. In the case of the trapezoid, suppose $AB = d_2$. Then the diagonal cannot be $d_2$ since both of the diagonals are adjacent to the edge of length $d_2$. If the diagonal were $d_1$, then $\triangle CDA$ and $\triangle DCB$ would be congruent equilateral triangles on the same side of the same base. Specifically, they would be coincident which would imply $A = B$, a clear contradiction. The only possible remaining cyclic quadrilateral is a square of side length $d_1$. Then, $AC$ and $BD$ are necessarily $d_2$. 
 		  
 		   An octahedron centered at the origin positioned along the coordinate axes is given by vertices $\{(\pm r, 0, 0), (0, \pm r, 0), (0, 0, \pm r)\}$ for some positive constant $r$. Assume without loss of generality that the shared midpoint of $\Oc E$, $AC$ and $BD$ lies at the origin and $E$ lies at $(d_2/2, 0, 0)$. This clearly presents no problem as it requires at worst a similarity transformation of any configuration satisfying the distances determined above. Then we must conclude that $\Oc = (-d_2/2, 0, 0)$. It is then easy to see that all of the remaining points must lie at the coordinates of the remaining vertices of the above octahedron. Then, we may conclude that any configuration of six points in three dimensions determining two distinct triangles must be given by the vertices of an octahedron.
 	\end{proof}
 	
 	With the base case established, we can prove the following partial result by induction on number of dimensions.
 	\begin{prop}\label{prop:unique}
 		The set of vertices of the $d$-orthoplex is the unique configuration of $2d$ points in $d\geq 3$ dimensions determining two distinct triangles.
 	\end{prop}
 	\begin{proof}
 	
 	 		Assume the unique point configuration determining two distinct triangles in $k$ dimensions is determined by the vertices of the $k$-orthoplex, which specifically contains $2k$ points. Now, consider a configuration of $2k+2$ points determining two distinct triangles in $k+1$ dimensions. By Proposition \ref{prop:elim_dds}, it must determine exactly two distinct distances, and by Proposition \ref{prop:tri_geom} one of the triangles must be equilateral, and the other is isosceles with its repeated edge length equal to the side length of the equilateral triangle. For clarity, relabel $P_{2k+1}$ as $Q$, then we may assume without loss of generality that $\Oc Q = d_2$. Notice that a single edge length of $d_2$ is sufficient to determine that a given triangle is congruent to $T_{iso}$, so specifically we must have $\Oc P_i = Q P_i = d_1$ for all $1 \leq i \leq 2k$. So, all of these points are equidistant from $\Oc$ and $Q$, so as mentioned before, they lie in a $k$-hyperplane. Notice further that since the distances determined by each of the points from $\Oc$ and $Q$ are the same, specifically, all of the points $P_i$ lie on a $(k-2)$-sphere centered on the midpoint of $\Oc Q$. 
 	 		
 	 		For the entire point set to determine two distinct triangles, these $2k$ points must determine either one or two distinct triangles. The greatest number of points in $k$ dimensions determining one distinct triangles is $k+1$ \cite{brenner_et_al}, so for $k>1$, $2k$ points in $k$ dimensions must determine more than one distinct triangle. So, specifically, the points $P_1, \dots, P_{2k}$ must determine two distinct triangles. Since they all lie in the same $k$-hyperplane, we may conclude from the inductive hypothesis that these points then must be the vertices of a $k$-orthoplex. If we then assume again that the center of the $k$-orthoplex lies at the origin and the remaining vertices lie along the $x_1,\dots, x_k$ coordinate axes, then clearly $\Oc$ and $Q$ lie on the $x_{k+1}$ coordinate axis. Specifically, they lie at coordinates $(0, \dots, 0, d_2/2)$ and $(0, \dots, 0, -d_2/2)$. So, the whole point configuration must determine the vertices of a $(k+1)$-orthoplex.
  \end{proof}
  
  \begin{cor}\label{cor:optimal}
  	Any configuration of strictly greater than $2d$ points in $d\geq 3$ dimensions determines more than two distinct triangles.
  \end{cor}
  \begin{proof}
  	Consider a configuration of $2d+1$ points in $d$ dimensions. As mentioned before, since at most $d+1$ points can determine a single triangle, any $2d$ points must determine two distinct triangles. Then, by Proposition \ref{prop:unique}, any $(2d)$-subset of the point configuration must be the vertices of a $d$-orthoplex. So, specifically, $P_1,\dots, P_{2d}$ must be given by the vertices of a $d$-orthoplex. For convenience, relabel $P_{2d+1}$ as $P'$ Consider the distance between $P'$ and an arbitrary vertex of the orthoplex. Each vertex of the orthoplex lies at distance $d_2$ from exactly one other vertex of the orthoplex, so $P'	$ cannot lie at distance $d_2$ from any of the vertices of the orthoplex; otherwise, an isosceles triangle with repeated edge $d_2$ would be created. So, the distance from $P'$ to each of the vertices of the orthoplex must be $d_1$. Clearly this is the center of the unique $(d-1)$-sphere on which the vertices of the orthoplex lie. This must coincide with the common midpoint of the segments of length $d_2$. Assume without loss of generality that $P_1$ is not opposite $P_2$. Then, $P'P_1 \perp P'P_2$. That is to say, $\triangle P'P_1 P_2$ is right. But, notice that $P_1 P_2 = d_1$, so it is also equilateral, which is a contradiction. So, a configuration of at least $2d+1$ points in $d$ dimensions necessarily determines more than two distinct triangles. So the point configuration consisting of the vertices of the $d$-orthoplex in $d$ dimensions is optimal.
  \end{proof}
  
  Then, Theorem \ref{thm:main_result} follows directly from Proposition \ref{prop:unique} together with Corollary \ref{cor:optimal}.
  
  \section{Proofs of Lemmas}
  For the main proof, we used the three essentially combinatorial lemmas, Lemma \ref{lem:max_dds}, Lemma \ref{lem:distinct_bound} and Lemma \ref{lem:repeat_bound}. In this section, we restate and prove these three lemmas.
  
  \lemmaxdds*
  \begin{proof}
		Let $S$ be a finite point configuration determining $t$ distinct triangles. Consider an arbitrary finite ascending chain $A_3 \subset A_4 \subset \dots \subset A_{\abs{S} - 1} \subset S$ such that $A_3 \subseteq S$ is a three point subset of $S$, and $A_{i+1} \setminus A_i$ contains a single point. Clearly, $A_3$ determines at most 3 distinct distances and 1 distinct triangle. Consider $A_{k+1}$ for some $k$, and let $Q\in A_{k+1}$ be the unique point in $A_{k+1}$ such that $Q\notin A_{k}$. Notice that if $T(A_{k}) = T(A_{k+1})$, then the number of distinct distances determined by $A_k$ must be the same as that determined by $A_{k+1}$. So, assume that $T(A_k) \neq T(A_{k+1})$. Notice that every possible representative of any element of $T(A_{k+1}) \setminus T(A_k)$ is of the form $\triangle Q P_i P_j$ where $P_i, P_j \in A_k$. Then, the only possible new distinct distances are $QP_i$ and $QP_j$, so the difference between the sizes of the distance sets of $A_{k+1}$ and $A_k$ is at most $2\cdot (\abs{T(A_{k+1})} - \abs{T(A_k)})$. Then, proceeding inductively, it is clear that the number of distinct distances determined by $S$ is at most $2(\abs{T(S)} - 1) + 3 = 2t + 1$.
  \end{proof}
  
  \lemdistinctbound*
  \begin{proof}
  	This result is essentially a direct consequence of a result by Fort and Hedlund (Theorem 1 in \cite{fort_hedlund}). For convenience, a paraphrasing of Fort and Hedlund's result is as follows. Given the set of all distinct unordered pairs of positive integers between 1 and $n$ inclusive, a covering of this set by triples is a set of distinct unordered triples such that each pair is a subset of at least one of the triples. The minimal number of triples needed to cover the set of distinct unordered pairs of distinct positive integers is $\ceil{n/3 \cdot \floor{n/2}}$. 
  	
  	Consider a point $Q$ and points $P_1, \dots, P_k$ such that $Q P_i = d_i$ are all distinct distances. Then, consider each pair of these distances. The corresponding points must determine a triangle $\triangle Q P_i P_j$. For the purposes of this proof, a triangle may be thought of as the unordered triple of its edge lengths. Then, the minimal number of triangles needed to realize this configuration is clearly the minimal number of unordered triples of distinct $d_i$'s required to cover all distinct unordered pairs of distinct $d_i$'s. This problem is clearly isomorphic to the problem considered by Fort and Hedlund (i.e., by replacing $d_i$ with $i$), so we can conclude that the minimal number of triangles determined by a point set containing such a point is $\ceil{n/3 \cdot \floor{n/2}}$. 
	\end{proof}  	
	
	\lemrepeatbound*
	\begin{proof}
		Although the setup differs somewhat from the setup above, this proof will proceed in a similar fashion. While Fort and Hedlund considered only pairs of distinct numbers (e.g., $(1,2)$, $(4, 1)$, etc.), for the following setup, we will also allow pairs of the same number (e.g., $(2,2)$, $(5,5)$). Then, consider a set of distinct pairs of (not necessarily distinct) integers between 1 and $n$ inclusive. Let $n$ be the number of pairs of distinct integers and $m$ be the number of pairs not distinct integers. We hope to produce a lower bound on the size of a minimal covering by triples of such an arbitrary set. 
		
		Given a triple $(a, b, c)$, clearly this triple can cover at most 3 distinct pairs, $(a, b)$, $(b, c)$ and $(a, c)$. So, if we have $m=0$, we clearly require at least $n/3$ triples to cover the $n$ pairs of distinct integers. Consider now $m\neq 0$. Each pair of non-distinct integers $(a, a)$ needs to be covered by its own distinct triple $(a, a, b)$. That is to say, $(a, a)$ and $(b, b)$ cannot be covered by the same triple.  Notice that the triple $(a, a, b)$ also includes the pair $(a, b)$. So, each triple required to cover a pair of non-distinct integers can also contain a pair of distinct integers. So one fewer of the distinct pairs needs to be covered separately. Thus, a lower bound on the total number of triples needed to cover such pairs is clearly $\ceil{(n-m)/3 + m} = \ceil{(n+2m)/3}$.
		
		We can apply the same analogy as used in the proof of Lemma \ref{lem:distinct_bound} to apply this bound to our problem. Notice that if a point determines $k$ distinct distances, it necessarily determines $k$ choose $2$ distinct pairs of distinct distances. Substituting this into our above bound, we finally come to the conclusion that given a point which determines $n$ distinct distances, $m$ of which are repeated, this point configuration determines at least $$\ceil*{\frac{n(n-1) + 4m}{6}}$$ distinct triangles.
	\end{proof}
	
	\appendix
	\section{Some Notes on Elementary Euclidean Geometry}
	\label{app:geom}
	
	Throughout the geometric portions of our proofs, we rely on a handful of facts about $n$-dimensional Euclidean geometry which are standard but nontrivial. In this section, we seek to state these facts explicitly and offer sketches of their proofs. 
	\begin{rek}
		A collection of $d+1$ points in general position in $d$ dimensions (that is to say, no $k+1$ points lie in a $k$-hyperplane) uniquely determines a $(d-1)$-sphere.
	\end{rek}
	\begin{proof}[Proof Sketch]
		The idea is to proceed inductively with a base case of the standard fact that three points uniquely determine a circle. Then consider the problem with 4 points in $\R^3$. Choose any two distinct sets of three of those points, which uniquely determine two distinct circles. Consider the unique lines normal to the planes of each of these circles that pass through their centers. Notice that all of the points on each line are equidistant from the three points determining their respective circles. So it is necessary only to observe that these lines intersect, as their point of intersection is clearly the center of the unique sphere containing the four points. The structure of this argument naturally generalizes to higher dimensions. 
	\end{proof}
	\begin{rek}
		The set of points equidistant from two points $A$ and $B$ in $\R^d$	is a $(d-1)$-hyperplane perpendicular to the line $\overline{AB}$. 
	\end{rek}
	\begin{proof}[Proof Sketch]
		Consider taking a point $P:(x_1,\dots, x_d)\in\R^d$ such that $P$ is equidistant from two fixed points $A$ and $B$. Simply write these points in coordinates and consider the equation $\text{d}(A, P) = \text{d}(P, B)$. Notice that after simplification, this yields a linear equation in $d$ variables, which is well-known to be the equation of a $(d-1)$-hyperplane.
	\end{proof}
	\begin{rek}
		The set of points at a fixed distance from two distinct points $A$ and $B$ in $\R^d$ is a $(d-2)$-sphere centered on the midpoint of $AB$.
	\end{rek}
	\begin{proof}[Proof Sketch]
		Let $S$ denote the set of all points at a fixed distance (greater than half the length of $AB$) from both $A$ and $B$. By the above remark, these points lie in a $(d-1)$-hyperplane perpendicular to the line joining $A$ and $B$. Note that the midpoint $M$ of $AB$ lies in the hyperplane. So, consider the right triangle formed by $A$, an arbitrary point in $S$ and $M$. By the hypotenuse-leg congruence theorem, clearly every point in $S$ is equidistant from $M$. Finally, it is clear we can reverse this argument to see that every point that is at the same fixed distance from $M$ must belong to $S$.
	\end{proof}
  	
  \bibliographystyle{alpha}
  \bibliography{biblio}
\end{document}